\documentclass[11pt,a4paper]{amsart}
\usepackage{amsmath,amssymb,amsthm}

\usepackage[dvips]{graphicx}

\theoremstyle{plain}
\newtheorem{theorem}{Theorem}[section]

\newtheorem{proposition}[theorem]{Proposition}
\newtheorem{cor}[theorem]{Corollary}

\newtheorem{example}[theorem]{Example}
\newtheorem{remark}[theorem]{Remark}

\newcommand{\bk}{\mbox{\boldmath $k$}}
\newcommand{\bl}{\mbox{\boldmath $\ell$}}

\title[Ehrenfest Diffusion Model]
{Finite Gelfand Pair Approaches for Ehrenfest Diffusion Model}
\author{Hiroshi Mizukawa}
\date{}
\address{Department of Mathematics, National Defense Academy of Japan,
 Yokosuka 239-8686, Japan.}
\email{mzh@nda.ac.jp}
\begin{document}
\maketitle
\footnote[0]{Classification number :20C30, 20P05, 05E10.}
\begin{abstract}
A classical diffusion model of Ehrenfest
which consists of $2$-urns and $n$-balls is realized by 
a finite Gelfand pair $(H_{n},S_{n})$,
where $H_{n}$ is the hyperoctahedral group and $S_{n}$ is the symmetric group. 
This fact can be generalized to multi-urn version by using 
Gelfand pairs of  complex reflection groups .
\end{abstract}
\section{introduction}
There are two urns,  the left one containing $n$-balls and the right one having no ball.
We shuffle the balls according to a rule as follows;
At each step one ball is chosen randomly and the ball is moved to the other urn chosen randomly.    
This  process is called {\it{Ehrenfest diffusion model}}. 

We consider a generalization of the above process by increasing the number of urns.
 Let $r \geq 2$. 
There are 
$r$ distinct urns $U_0,U_1,\cdots,U_{r-1}$, initially $U_{0}$ contains  distinct balls $B_{1},\cdots,B_{n}$.
Let us consider a similar diffusion process as above:
At each step, we choose a ball randomly and transfer the ball into one of the other urns.
Here we consider some kind of destinations of picked balls. 
The first one is  any other urns without where it was (Section 3).
Another two cases are considered in the Appendix of this paper; 
The second case is that a ball  picked from $U_{i}$ is moved to $U_{i+1\pmod{r}}$ (Appendix \ref{32}).
The last case is that a ball picked from $U_{i}$ is moved to $U_{i+1\pmod{r}}$ or $U_{i-1\pmod{r}}$ (Appendix \ref{33}).

If $U_{j}$ contains the ball $B_{i}$, then we 
define a function on the balls by $b_{i}=b(B_{i})=j$.
Then we can identify each configuration of our model 
 with an element of a set
$$B(r,n)=\{(b_{i}\mid 1 \leq i \leq n)\mid 0\leq b_{i} \leq r-1\}.$$
In the next section, we see that
the set  $B(r,n)$ is an realization of a finite homogenous space $G(r,1,n)/S_{n}$,
where $G(r,1,n)$ is an complex reflection group and $S_{n}$ is a symmetric group.
Indeed this pair of groups $(G(r,1,n),S_{n})$ is a Gelfand pair.
Our main purpose is to analyze a stochastic space $B(r,n)$  by using this Gelfand pair.
Further the book \cite{dp} is good introduction for an application of finite Gelfand pairs to 
probability theory. 
\section{$B(r,n)$ and $(G(r,1,n),S_{n})$}
In this section we introduce how to identify $B(r,n)$ with a certain finite 
and discrete homogenous space.
Let $S_{n}$ be the symmetric group 
and  $G(r,1,n)={\mathbb Z}/r{\mathbb Z} \wr S_{n}$  the complex reflection group.
We denote an element of $G(r,1,n)$ by $(x_{1},\cdots,x_{n};\sigma)$,
where $x_{i} \in {\mathbb Z}/r{\mathbb Z} $ and $\sigma \in S_{n}$.
Under this notation, we remark that $S_{n}$ is a subgroup $\{(0,\cdots,0;\sigma )\mid \sigma \in S_{n}\}$ of $G(r,1,n)$. 
We define an action of $G(r,1,n)$
on $B(r,n)$ by
$$x (b_{i}\mid 1 \leq i \leq n)=(x_{i}+b_{\sigma^{-1}(i)}\pmod{r}\mid 1 \leq i \leq n),$$
where $x=(x_{1},\cdots,x_{n};\sigma) \in G(r,1,n)$ and $(b_{i}\mid 1 \leq i \leq n) \in B(r,n)$.
\begin{proposition}
This action is transitive on $B(r,n)$.
\end{proposition}
\begin{proof}
We set $x=(r-b_{1},\cdots,r-b_{n};1)$. Then we have
$x(b_{i}\mid1 \leq i \leq n)=(0,\cdots,0)$. Furthermore this action is invertible.
Therefore any two elements are transferred to each other by the action of $G(r,1,n)$.  
\end{proof}

\begin{proposition}
Put $I_{0}=(0,\cdots,0) \in B(r,n)$.
Then the stabilizer of $I_{0}$ is $S_{n}$.
\end{proposition}
\begin{proof}
It is clear that the definition of the action.
\end{proof}
From these propositions, we can identify $B(r,n)$ with a finite homogenous space $G(r,1,n)/S_{n}$.
Let ${\mathbb N}=\{0,1,2,\cdots\}$. 
For a composition $\bold{m} =(m_{0},\cdots,m_{r-1})$ of $n$,
let $|\bold{m} |=m_{0}+\cdots+m_{r-1}$ be the size of $\bold{m} $.
Put 
$N(r,n)=\{\bold{m} \in {\mathbb{N} }^r\mid |\bold{m}|=n\}$.
For  $\bold{m} \in N(r,n)$,
we define a partition by 
$\lambda(\bold{m})=(0^{m_{0}}1^{m_{1}}\cdots(r-1)^{m_{r-1}})$, i.e, 
$m_{i}$ is regarded as the multiplicity of $i$.
Let $m_{\lambda}(x_{1},\cdots,x_{n})$ be a monomial symmetric polynomial indexed by a partition $\lambda$.
Let $\xi$ be a primitive $r$th  root of unity. 
We set $$m_{\lambda}(\bl)=m_{\lambda}(\underbrace{1,\cdots,1}_{\ell_{0}},
\underbrace{\xi,\cdots,\xi}_{\ell_{1}},\cdots,
\underbrace{\xi^{r-1},\cdots,\xi^{r-1}}_{\ell_{r-1}}),$$
for $\bl=(\ell_{0},\cdots,\ell_{r-1})\in N(r,n)$.

Then the following theorem holds.
\begin{theorem}\label{mzonal}\cite{miz}
\begin{enumerate}
\item
A pair $(G(r,1,n),S_{n})$ is a Gelfand pair.
\item
The permutation representation is decomposed as
$${\mathbb C}G(r,1,n)/S_{n}\sim \bigoplus_{\bk \in  N(r,n)}V(\bk),$$
where $V(\bk)$ is an irreducible representation of $G(r,1,n)$ with $\dim V(\bk)=\binom{n}{k_{0},\cdots,k_{r-1}}$.
\item 
Let $\omega_{\bk}$ be the zonal spherical function corresponding to $V(\bk)$.
For $x=(x_{1},\cdots,x_{n};\sigma) \in G(r,1,n)$, we have
$$\omega_{\bk}(x)=\frac{m_{\lambda(\bk)}(\xi_{1},\cdots,\xi_{n})}{\binom{n}{k_{0},\cdots,k_{r-1}}},$$
where $\xi_{i}=\xi^{x_{i}}$.
Moreover, the  table of the zonal spherical functions is given by
$$\left(\frac{m_{\lambda(\bk)}(\bl)}{\binom{n}{k_{0},\cdots,k_{r-1}}}\right)_{\bk,\bl \in N(r,n)}.$$
\end{enumerate}
\end{theorem}
We denote by $\omega_{\bk,\bl}=\frac{m_{\lambda(\bk)}(\bl)}{\binom{n}{k_{0},\cdots,k_{r-1}}}$ the value of 
zonal spherical function indexed by $(\bk,\bl)$.
\section{$B(r,n)$ and $(G(r,1,n),S_{n})$ }
We consider a stochastic processes on $B(r,n)$.
The initial distribution is given by
$\nu_{0}(b)=\begin{cases}
1,& b=I_{0},\\
0,& b\not=I_{0}.
\end{cases}$ 
Let $\pi$ be the uniform distribution on $B(r,n)$, i.e. $\pi\equiv \frac{1}{r^n}$.
Similarly we denote by $\tilde{\pi}$ the uniform distribution on $G(r,1,n)$.

Let $P=(p(a,b))_{a,b \in B(r,n)}$ be a $G(r,1,n)$-invariant stochastic 
matrix on $B(r,n)$ i.e., $p(ga,gb)=p(a,b)$ for any $a,b \in B(r,n)$ and $x \in G(r,1,n)$.
Then we define a function $\nu$ on $B(r,n)$ by
$\nu(b)=p(b_{0},b)$. 
Since the action of $G(r,1,n)$ is transitive,
there exists $g \in G(r,1,n)$ such that $b=gb_{0}$. 
Put $\tilde{\nu}(g)=\frac{1}{n!}p(b_{0},g b_{0})$.
It is easy to check that $\nu$ is a stochastic distribution  and 
a bi-$S_{n}$ invariant function on $G(r,1,n)$.
Therefore $\tilde{\nu}$ can be expanded by the zonal spherical functions, say
$\tilde{\nu}=\sum_{\bk \in N(r,n)} a_{\bk} \omega_{\bk}$. 
Then the orthogonality relation of the zonal spherical functions gives us the following proposition. 
%
\begin{proposition}
Let $\bk=(k_{0},\cdots,k_{r-1})\in N(r,n)$.
Put $f_{\bk}=\sum_{g \in G(r,1,n)}\tilde{\nu}(g)\overline{\omega_{\bk}(g)}$. Then the coefficients $a_{k}$'s are 
expressed by 
$$a_{\bk}=\frac{{\binom{n}{k_{0},\cdots,k_{r-1}}}}{r^n n!}f_{\bk}.$$
\end{proposition}
%
\begin{proof}
Take an inner product  of the both sides of $\tilde{\nu}$.
\end{proof}
The probability being in a state $b=gb_{0}\ (g \in G(r,1,n))$ after $N$-steps iterate 
with a start point $b_{0}$ is 
\begin{align*}
p^{(N)}(b)&=\sum_{b_{0},\cdots,b_{k-1}\in B(r,n)}p(b_{0},b_{1})p(b_{1},b_{2})\cdots p(b_{N-1},b)\\
&=
\sum_{b_{0},\cdots,b_{N-1}\in B(r,n)}p(b_{0},g_{1}b_{0})p(g_{1}b_{0},g_{2}b_{0})\cdots p(g_{N-1}b_{0},gb)\\
&=\sum_{x_{1},\cdots,x_{N-1} \in G(r,1,n)}\tilde{\nu}(x_{1})\tilde{\nu}(x_{1}^{-1}x_{2})\cdots \tilde{\nu}(x_{N-1}^{-1}g)
=\tilde{\nu}^{*N}(g) \ (N{\rm th.\ convolution\ power}),
\end{align*}
where we denote by $g_{i}$ an element satisfying $b_{i}=g_{i}b_{0}$.
By using a property of the  zonal spherical functions $\omega_{\bk}*\omega_{\bk'}=\frac{|G(r,1,n)|}{\dim V(\bk)}
\delta_{\bk \bk'}$, we have
$$\tilde{\nu}^{*N}=\sum_{\bk \in N(r,n)}a_{\bk}^{N} \omega_{\bk}^{*N}=
\sum_{\bk \in N(r,n)}\frac{{\binom{n}{k_{0},\cdots,k_{r-1}}}}{r^n n!}f_{\bk}^{N}\omega_{\bk}.$$
For 
stochastic distributions $\mu$ and $\mu'$ on a space $X$,
the total variation distance is defined by 
$$||\mu-\mu'||_{TV}=\frac{1}{2}\sum_{x \in X}|\mu(x)-\mu'(x)|.$$ 
Then the following estimate is known.
\begin{proposition}\cite[Corollary 4.9.2, pp. 144]{cst}\label{ests}
$$||\nu^{*N}-\pi||^2_{TV}=||\tilde{\nu}^{*N}-\tilde{\pi}||^2_{TV} \leq \frac{1}{4}\sum_{}\binom{n}{k_{0},\cdots,k_{r-1}}|f_{k}|^{2N},$$
where ${\bk}=(k_{0},\cdots,k_{r-1})$ runs over $N(r,n)$ except ${\bk}=(n,0,\cdots,0)$ which  corresponds to the trivial representation of $G(r,1,n)$.
\end{proposition}
We compute $f_{{\bk}}$'s for three distinct stochastic matrices on $B(r,n)$. 
We define a function on $B(r,n)\times B(r,n)$ by
$$d_{r}(b,c)=\#\{i\mid b_{i}\not=c_{i}\},$$
where $b=(b_{i}\mid 1 \leq i \leq n)$ and $c=(c_{i}\mid 1 \leq i \leq n)$.
We define a stochastic matrix by
$$p(x,y)=\begin{cases}
\frac{1}{(r-1)n},& d_{r}(x,y)=1\\
0,& d_{r}(x,y)\not=1.
\end{cases}$$
Clearly the setting means the first way of shuffle explained in Section 1.
\begin{proposition}
For $\bk=(k_{0},\cdots,k_{r-1})\in B(r,n)$, we have
$f_{\bk}=\frac{1}{r-1}\left(\frac{rk_{0}}{n}-1\right)$
\end{proposition}
\begin{proof}
We define 
$g_{m,j}=(g_{1},\cdots,g_{n}:1) \in G(r,1,n)$
by 
$g_{i}=\begin{cases}
j & (i=m),\\
0 & (i \not= m)
\end{cases}$
for any $1 \leq m \leq n$ and $1 \leq j \leq r-1$.
Then we have
$ \nu(g)=\begin{cases}
1 & (g=g_{m,j}),\\
0 & (g\not= g_{m,j}).
\end{cases}$
Now we can compute $f_{\bk}$ as follows.
\begin{align*}
\sum_{g \in G(r,1,n)}\tilde{\nu}(g)\overline{\omega_{{\bk}}(g)}
&=\frac{1}{n(r-1)}\sum_{j=1}^{r-1}
\sum_{m=1}^n
\overline{\omega_{\lambda({\bk})}(g_{m,j})}\\
&=\frac{1}{r-1}\sum_{j=1}^{r-1}\frac{\sum_{i=0}^{r-1}\binom{n-1}{k_{0},\cdots,{k_{i}-1},\cdots,k_{r-1} } \zeta^{-ij}}
{\binom{n}{k_{0},\cdots,k_{r-1}}}
=\frac{1}{r-1}\sum_{i=0}^{r-1}\frac{k_{i}}{n}\sum_{j=1}^{r-1}\zeta^{-ij}\\
&=\frac{1}{r-1}\left\{\frac{k_{0}(r-1)}{n}+\sum_{i=1}^{r-1}\frac{-k_{i}}{n}\right\}
=\frac{1}{r-1}\left(\frac{rk_{0}}{n}-1\right).
\end{align*}
In the second equality we use Theorem \ref{mzonal} (3) and the following equation;
$$m_{\lambda(\bk)}(1,\cdots,1,x)=\sum_{i=0}^{r-1}\binom{n-1}{k_{0},\cdots,{k_{i}-1},\cdots,k_{r-1} } x^{i}.$$
\end{proof}
From the proposition above, we can easily  show $|f_{k}|\leq 1$
and $|f_{k}|=1 \Leftrightarrow 
k_{0}=
\begin{cases}
n& (r \geq 3)\\
0\ {\rm or}\ n& (r=2).
\end{cases}$
\begin{theorem}
For $g =(x_{1},\cdots,x_{n}:\sigma)$, we have
$$\tilde{\nu}^{*N}(g)=\frac{1}{r^n n!}\sum_{k_{0}=0}^n\left( \frac{rk_{0}-n}{n(r-1)}\right)^Ne_{n-k_{0}}(\Phi_{1},\cdots,
\Phi_{n}),$$
where $\Phi_{i}=\xi^{x_i}+\xi^{2x_{i}}+\cdots+\xi^{{(r-1)}x_{i}}=
\begin{cases}
r-1&(x_{i}=1),\\
-1&(x_{i}\not=1)
\end{cases}$ and $e_{j}(x_{1},\cdots,x_{n})$ is the $j$-th elementary symmetric polynomial.  
\end{theorem}

\begin{proof}
For $\bk=(k_{0},k_{1},\cdots,k_{r-1})$, we have the following generating function (\cite{miz})
$$\prod_{i=1}^{n}(1+\Phi_{i})=\sum_{k}\binom{n}{k_{0},\cdots,k_{r-1}}\omega_{\bk}(g).$$ 
We remark that 
$\prod_{i=1}^{n}(1+\Phi_{i})=\sum_{j}e_{j}(\Phi_{1},\cdots,\Phi_{n})$
and
$e_{j}(\Phi_{1},\cdots,\Phi_{n})=\sum_{\ell(\lambda(\bk))=j} m_{\lambda(\bk)}(\xi_{1},\cdots,\xi_{n})$,
where $e_{j}$ is an elementary symmetric polynomial 
and $\ell(\lambda)$ is the length of $\lambda$. 
We compute	
\begin{align*}
\tilde{\nu}^{*N}&=\frac{1}{r^n n!}\sum_{\bk \in N(r,n)}\frac{\binom{n}{k_{0},\cdots,k_{r-1}}}{r^n n!}
f_{\bk}^N\omega_{\bk}\\
&=\frac{1}{r^n n!} \sum_{k_{0}=0}^{n}
\binom{n}{k_{0}}
\left( \frac{rk_{0}-n}{n(r-1)}\right)^N
\sum_{k_{1}+\cdots+k_{r-1}=n-k_{0}}
\binom{n-k_{0}}{k_{1},\cdots,k_{r-1}}\omega_{k_{0},\cdots,k_{r-1}}(g)\\
&=
\frac{1}{r^n n!} \sum_{k_{0}=0}^{n}
\left( \frac{rk_{0}-n}{n(r-1)}\right)^N
\sum_{k_{1}+\cdots+k_{r-1}=n-k_{0}}
m_{\lambda(\bk)}(\xi_{1},\cdots,\xi_{n})\\
&=
\frac{1}{r^n n!} \sum_{k_{0}=0}^{n}
\left( \frac{rk_{0}-n}{n(r-1)}\right)^N
e_{n-k_{0}}(\Phi_{1},\cdots,
\Phi_{n}).
\end{align*}
\end{proof}
From this theorem, we have immediately the following corollary. 
\begin{cor}\label{2}
If $r>2$, then 
$\lim_{N \rightarrow \infty }\tilde{\nu}^{*N}=\frac{1}{r^n n!}.$
If $r=2$, let $\ell$ be a number of balls in the urn 1, then  
$$\lim_{N \rightarrow \infty }\tilde{\nu}^{*2N}=
\begin{cases}
\frac{1}{2^{n-1} n!} & (\ell \equiv 0 \pmod{2}),\\
0 &(\ell \equiv 1 \pmod{2})
\end{cases}
,\ \ \ 
\lim_{N \rightarrow \infty }\tilde{\nu}^{*2N-1}=
\begin{cases}
\frac{1}{2^{n-1} n!} &(\ell \equiv 1 \pmod{2}),\\
0 & (\ell \equiv 0 \pmod{2})
\end{cases}.$$
\end{cor}
We try to estimate an upper bound of $\tilde{\nu}^{*N}$.
Before state a theorem,  we see the following example. 
\begin{example}
We put $r=3$ and $n=20$. Then we have the following the graph of 
total variation distance. 
\begin{center}{
\includegraphics[width=12cm]{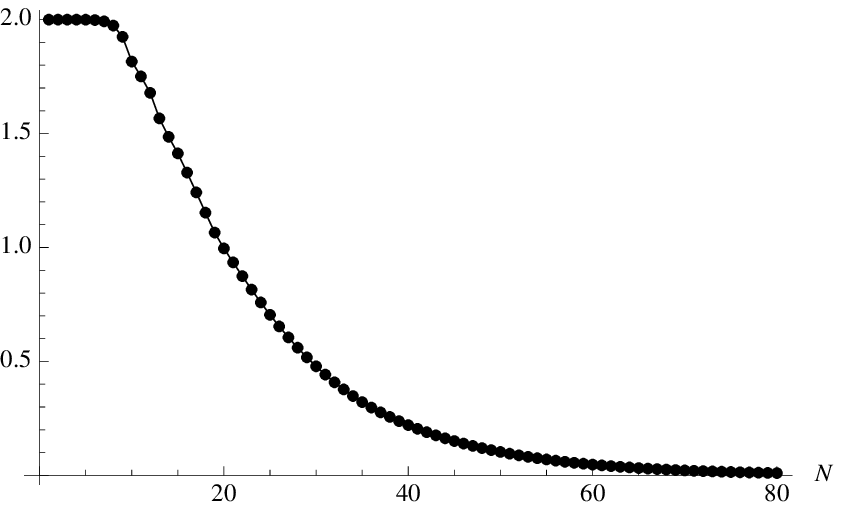}
}\end{center}
Here The horizontal axis is  the number of shuffles $N$.
\end{example}
\begin{theorem}
Put $N=\frac{n(r-1)}{2r}(\log r^n+c)$.
\begin{enumerate}
\item
In the case of $r=2$, we have
$$||\nu^{*N}-\pi ||_{TV}^2-1/4 \leq \frac{1}{4}e^{-c}.$$
\item
In the case of $r\geq 3$, we have
$$||\nu^{*N}-\pi ||_{TV}^2 \leq \frac{1}{4}e^{-c}.$$
\end{enumerate}
\end{theorem}
\begin{proof}
We refer to Proposition \ref{ests} and compute
\begin{align*}
||\nu^{*N}-\pi||_{TV}^2&\leq \frac{1}{4}\sum_{\bk\in N(r,k),k_{0}\not=n}
\binom{n}{k_{0},\cdots,k_{r-1}}\left|\frac{1}{r-1}\left(\frac{rk_{0}}{n}-1\right)\right|^{2N}\\
&=\frac{1}{4}\sum_{k_{0}=0}^{n-1}\binom{n}{k_{0}}\left|\frac{1}{r-1}\left(\frac{rk_{0}}{n}-1\right)\right|^{2N}\sum_{k_{1}+\cdots
+k_{r-1}=n-k_{0}}\binom{n-k_{0}}{k_{1},\cdots,k_{r-1}}\\
&=\frac{1}{4}\sum_{k_{0}=0}^{n-1}\binom{n}{k_{0}}\left|\frac{1}{r-1}\left(\frac{rk_{0}}{n}-1\right)\right|^{2N}(r-1)^{n-k_{0}}\\
&\leq
\begin{cases}
\frac{1}{4}+\frac{1}{4}e^{-\frac{4N}{n}}\sum_{k_{0}=1}^{n-1}\binom{n}{k_{0}}=\frac{1}{4}+\frac{1}{4}(2^n-2)e^{-\frac{4N}{n}}
\leq \frac{1}{4}+\frac{1}{4}2^ne^{-\frac{4N}{n}} & (r=2),\\
\frac{1}{4}e^{-\frac{2rN}{n(r-1)}}\sum_{k_{0}=0}^{n-1}\binom{n}{k_{0}}(r-1)^{n-k_{0}}=\frac{1}{4}(r^n-1)e^{-\frac{2rN}{n(r-1)}}
\leq \frac{1}{4}r^ne^{-\frac{2rN}{n(r-1)}} & (r \geq 3).
\end{cases}
\end{align*}
Here we use $|1-a| \leq e^{-a}$ for $a\leq1$.
\end{proof}
\begin{remark}
In the Theorem above, ``$1/4$" is the limit of $||\nu^{*N}-\pi ||_{TV}^2$ as $n \rightarrow \infty$ which comes from  a consequence of Corollary  \ref{2}.
\end{remark}
\section{Appendix}
Through this section, the urns $U_{0},U_{1},\cdots,U_{r-1}$ form a circle. 
Here we take up another shuffles and compute their Fourier coefficients.
\subsection{cyclic shuffle 1}\label{32}
We consider the following type of shuffle;  
``The ball picked randomly is transferred into the next urn on the left."
We define the following stochastic matrix on $B(r,n)$;
$$p(b,c)=\begin{cases}
\frac{1}{n},&{d_{r}}(b,c)=1\ {\rm and}\ b_{i}- c_{i}\in \{0,1,r-1\}\ (1 \leq \forall i \leq n),\\
0,& {\rm otherwise}.
\end{cases}$$
Clearly this gives us the probabilities of the shuffle introduced above.
Now we have 
\begin{proposition}
For $k=(k_{0},\cdots,k_{r-1})\in B(r,n)$, we have
$f_{k}=\sum_{i=0}^{r-1}\frac{k_{i}}{n}\zeta^{-i}.$
\end{proposition}
\begin{proof}

We define 
$g_{m,j}=(g_{1},\cdots,g_{n}:1) \in G(r,1,n)$
by 
$g_{i}=j \delta_{im}$.
\begin{align*}
\sum_{g \in G(r,1,n)}\tilde{\nu}(g)\overline{\omega_{\bk}(g)}
&=\frac{1}{n}
\sum_{m=1}^n
\overline{\omega_{\lambda(\bk)}(g_{m,1})}\\
&=\frac{\sum_{i=0}^{r-1}\binom{n-1}{k_{0},\cdots,{k_{i}-1},\cdots,k_{r-1} }
\zeta^{-i}}
{\binom{n}{k_{0},\cdots,k_{r-1}}}
=\sum_{i=0}^{r-1}\frac{k_{i}}{n}\zeta^{-i}.
\end{align*}
\end{proof}
\subsection{cyclic shuffle 2}\label{33}
Here a shuffle,
``The ball picked randomly is transferred into the next urn on the  right or left randomly,"
is considered.
The stochastic matrix on $B(r,n)$ of this shuffle is given by
$$p(b,c)=\begin{cases}
\frac{1}{2n},&{d_{r}}(b,c)=1\ {\rm and}\ b_{i}- c_{i}\in \{0,\pm1,\pm(r-1)\}\ (1 \leq \forall i \leq n),\\
0,& {\rm otherwise}.
\end{cases}$$
Now we have 
\begin{proposition}
For $k=(k_{0},\cdots,k_{r-1})\in B(r,n)$, we have
$f_{k}=\frac{1}{r-1}\left(\frac{rk_{0}}{n}-1\right)$
\end{proposition}
\begin{proof}

We define 
$g_{m,j}=(g_{1},\cdots,g_{n}:1) \in G(r,1,n)$
by 
$g_{i}=j \delta_{im}$.
\begin{align*}
\sum_{g \in G(r,1,n)}\tilde{\nu}(g)\overline{\omega_{\bk}(g)}
&=\frac{1}{2n}
\sum_{m=1}^n
\left(\overline{\omega_{\lambda(\bk)}(g_{m,1})}+\overline{\omega_{\lambda(k)}(g_{m,r-1})}\right)\\
&=\frac{1}{2}\frac{\sum_{i=0}^{r-1}\binom{n-1}{k_{0},\cdots,{k_{i}-1},\cdots,k_{r-1} }
\left(\zeta^{-i}+\zeta^{i}\right)}
{\binom{n}{k_{0},\cdots,k_{r-1}}}
=\frac{1}{2}\sum_{i=0}^{r-1}\frac{k_{i}}{n}\left(\zeta^{-i}+\zeta^{i}\right).
\end{align*}
\end{proof}

\end{document}